\documentclass{amsart}

\usepackage{amsmath,amsthm,amssymb}
\usepackage{./geoff}

\title[The Mackey Machine for Groupoid Crossed Products]{The Mackey Machine
  for Crossed Products by Regular Groupoids. II}
\author{Geoff Goehle}
\address{Mathematics and Computer Science Department, Stillwell 426,
  Western Carolina University, Cullowhee, NC 28723}
\email{grgoehle@email.wcu.edu}

\begin{document}

\begin{abstract}
We prove that given a regular groupoid $G$ whose isotropy subgroupoid
$S$ has a Haar system, along with a dynamical system $(A,G,\alpha)$,
there is an action of $G$ on the spectrum of $A\rtimes S$ such
that the spectrum of $A\rtimes G$ is homeomorphic to the orbit space
of this action via induction.  
In addition, we give a strengthening of these results in
the case where the crossed product is a groupoid algebra.  
\end{abstract}

\maketitle

\section*{Introduction}

This paper continues the development of the Mackey Machine for
groupoid crossed products which was started in \cite{inducpaper}.  In
the first paper of this series we constructed an induction process for
groupoid crossed products and proved that for crossed products by
regular groupoids every irreducible representation of $A\rtimes G$ is
induced from a representation of a ``stabilizer'' crossed product
$A(u)\rtimes S_u$.

In this work we realize our ultimate goal of identifying the space of
irreducible representations of certain crossed products by exhibiting 
a natural action of $G$ on the spectrum
$(A\rtimes S)\sidehat$, showing that induction defines a map from the spectrum
of $A\rtimes S$ onto the spectrum of $A\rtimes G$, and then proving
that this map factors to a homeomorphism
between the orbit space $(A\rtimes S)\sidehat/G$ and $(A\rtimes
G)\sidehat$.  This identification theorem 
is a partial generalization of
work done by Williams for transformation group $C^*$-algebras
\cite{primtrangroup} and is also related to work done by Orloff Clark
on groupoid $C^*$-algebras \cite{ccrgca,principclark}.
An outline of the paper is roughly as follows. 
Section \ref{sec:group-cross-prod} covers
some basic crossed product theory, as well as a few facts concerning crossed
products by groupoid group bundles.  Section \ref{sec:cross-prod-with}
contains the main result of the paper.  The proof is quite technical
and has been broken up into four subsections.  We finish with Section
\ref{sec:groupoid-algebras} which strengthens the results of Section
\ref{sec:cross-prod-with} in the context of groupoid algebras.

Before we begin in earnest we should first make some remarks about our
hypotheses.   In order to work with the crossed product
$A\rtimes S$ we must assume that $S$ has a Haar system.  It is worth 
pointing out that this is equivalent to assuming that the
stabilizer subgroups $S_u$ vary continuously with respect to the Fell
topology in $S$ \cite{renaultgcp}.  Finally, it should be noted that
to a large extent
the results of this paper are contained, with more detail and a great deal of
background material, in the author's thesis \cite{mythesis}.

\section{Preliminaries}
\label{sec:group-cross-prod}

We will be using the same notation and terminology as
\cite{inducpaper}.  In particular, 
we will let $G$ denote a second countable,
locally compact Hausdorff groupoid with a Haar system $\lambda$.
Given an element $u\in G\unit$ of the unit space of $G$ we
will use $S_u=\{\gamma\in G : s(\gamma)=r(\gamma)=u\}$ 
to denote the stabilizer, or isotropy, subgroup of $G$
over $u$.  We use $S=\{\gamma\in G : s(\gamma) = r(\gamma)\}$ to
denote the stabilizer, or isotropy, subgroupoid of $G$ formed by
bundling together all of the stabilizer subgroups.  
We will let $A$ denote a separable
$C_0(G\unit)$-algebra and will let $\mcal{A}$ be its associated
usc-bundle.  Given $A$ and $G$ as above
we let $\alpha$ denote an action of $G$ on $A$ as defined in
\cite[Definition 4.1]{renaultequiv} and call $(A,G,\alpha)$ a groupoid
dynamical system.  We 
construct the groupoid crossed product $A\rtimes_\alpha G$ as a universal
completion of the algebra of compactly supported sections
$\Gamma_c(G,r^*\mcal{A})$ in the usual fashion.

One important aspect of groupoid dynamical systems is that given
$(A,G,\alpha)$ there is a natural action of $G$ on the spectrum of $A$
induced by $\alpha$. 

\begin{prop}
\label{prop:3}
If $(A,G,\alpha)$ is a groupoid dynamical system then there is a
continuous action of $G$ on $\widehat{A}$ given by $\gamma\cdot \pi =
\pi\circ \alpha_\gamma\inv$.  
\end{prop}

\begin{proof}
Since $A$ is a $C_0(G\unit)$-algebra it follows from \cite[Proposition
C.5]{tfb2} that there is a continuous map $r:\widehat{A}\rightarrow
G\unit$.  Furthermore, we view
$\widehat{A}$ as being fibred over $G\unit$ so that if
$\pi\in\widehat{A}$ with $r(\pi) = u$ then 
we can factor $\pi$ to a representation $\pi'$ of $A(u)$. Given
$\gamma\in G$ we know 
$\alpha_\gamma:A(s(\gamma))\rightarrow A(r(\gamma))$ so that if
$r(\pi) = s(\gamma)$ we can define $\gamma\cdot \pi\in \widehat{A}$ by
$\gamma\cdot \pi(a) = \pi'(\alpha_\gamma\inv(a(r(\gamma))))$.
Of course, when we factor $\gamma\cdot \pi$ to $A(r(\gamma))$ we get
$(\gamma\cdot \pi)' = \pi'\circ\alpha_\gamma\inv$ as desired.  
The difficult part in proving that this defines a groupoid
action is showing that it is continuous. 

Suppose
$\gamma_i \rightarrow \gamma$ and  $\pi_i\rightarrow \pi$ such that
$s(\gamma_i) = r(\pi_i)$ for all $i$ and $s(\gamma) = r(\pi)$.
Let $O_J = \{\rho\in\widehat{A}: J\not\subset \ker\rho\}$ be an open set in
$\widehat{A}$ containing $\gamma\cdot \pi$.  Suppose, to the contrary, 
that $\gamma_i\cdot
\pi_i$ is not eventually in $O_J$.  By passing to a subnet and
relabeling  we can assume $\gamma_i\cdot\pi_i\not\in O_J$ for all
$i$.  Fix $a\in J$ and choose $b\in A$ such that $b(s(\gamma)) =
\alpha_\gamma\inv(a(r(\gamma)))$.  Since the action is continuous,
$\alpha_{\gamma_i}\inv(a(r(\gamma_i))) \rightarrow b(s(\gamma))$.
Since the norm is upper-semicontinuous, the set
$\{ a\in \mcal{A} : \|a\| < \epsilon\}$ is open for all $\epsilon >
0$.  Because $\alpha_{\gamma_i}\inv(a(r(\gamma_i))) - b(s(\gamma_i))
\rightarrow 0$, we eventually have
$\|\alpha_{\gamma_i}\inv(a(r(\gamma_i))) - b(s(\gamma_i))\| <
\epsilon$ for all $\epsilon >0$.  Hence
$\|\alpha_{\gamma_i}\inv(a(r(\gamma_i))) - b(s(\gamma_i))\| \rightarrow
0.$
Next, $\gamma_i\cdot \pi_i\not\in O_J$ for all $i$ so that
$\gamma_i\cdot \pi_i(a) = \pi'(\alpha_{\gamma_i}\inv(a(r(\gamma_i)))) =
0$ for all $i$.  Thus
\begin{equation}
\|\pi_i(b)\| = \|\pi'_i(b(s(\gamma_i))
- \alpha_{\gamma_i}\inv(a(r(\gamma_i))))\| 
\leq \|b(s(\gamma_i))-\alpha_{\gamma_i}\inv(a(r(\gamma_i)))\|
\rightarrow 0. \label{eq:65}
\end{equation}
It is shown in \cite[Lemma A.30]{tfb} that the map
$\pi\mapsto\|\pi(b)\|$ is lower-semicontinuous on $\widehat{A}$.  
In other words, given $\epsilon \geq 0$ the set 
$\{ \rho\in \widehat{A} : \|\rho(b)\|\leq \epsilon \}$
is closed.  Thus \eqref{eq:65} implies
that eventually
$\pi_i\in\{\rho\in\widehat{A}:\|\rho(b)\|\leq\epsilon\}$.  Therefore, 
the fact that $\pi_i\rightarrow \pi$ implies $\|\pi(b)\|\leq
\epsilon$.  This is true for all $\epsilon > 0$ so that 
\[
0 = \pi(b) = \pi'(b(s(\gamma))) =
\pi'(\alpha_\gamma\inv(a(r(\gamma)))) = \gamma\cdot \pi(a).
\]
This is a contradiction since $a\in J$ was arbitrary and we assumed
that $\gamma\cdot \pi\in O_J$.  
\end{proof}
\subsection{Bundle Crossed Products}

An important class of groupoids are those for which the range and
source map are identical.  Such a space is called a (groupoid) group
bundle and we will use $p$ to denote both the range and the source.
The premier example of a groupoid group bundle is the stabilizer
subgroupoid $S$ of a groupoid $G$.  The reason this class of
groupoids is important for what follows is that crossed products by
group bundles have extra structure.  

\begin{prop}
\label{prop:65}
Suppose $(A,S,\alpha)$ is a groupoid dynamical system and
$S$ is a group bundle.  Then $A\rtimes_\alpha S$ is a
$C_0(S\unit)$-algebra with the action defined for $\phi\in
C_0(S\unit)$ and $f\in\Gamma_c(S,p^*\mcal{A})$ by 
$\phi\cdot f(s) := \phi(p(s))f(s)$.
Furthermore, the restriction map from 
$\Gamma_c(S,p^*\mcal{A})$ to $C_c(S_u,A(u))$ factors to an
isomorphism of $A\rtimes_\alpha S(u)$ onto
$A(u)\rtimes_{\alpha|_{S_u}} S_u$.  
\end{prop}

\begin{proof}
Given $\phi\in C_0(S\unit)$ and $f\in \Gamma_c(S,p^*\mcal{A})$ define
$\Phi(\phi)f=\phi\cdot f$ as in the statement of the proposition.  
It is easy to see
that $\Phi(\phi)f\in \Gamma_c(S,p^*\mcal{A})$ and that $\Phi(\phi)$ is
linear as a function on $\Gamma_c(S,p^*\mcal{A})$. We need to extend
$\Phi(\phi)$ to an element of the multiplier algebra. First, simple
calculations show that, on $\Gamma_c(S,p^*\mcal{A})$, 
$\Phi(\phi)$ is $A\rtimes S$-linear and is adjointable with adjoint
$\Phi(\overline{\phi})$.  

Now extend $\Phi$ to the unitization $C_0(S\unit)^1$ by setting
$\Phi(\phi+\lambda 1)f = \Phi(\phi)f + \lambda f$.  An elementary
computation shows that $\Phi$ preserves the operations on
$C_0(S\unit)^1$.
Suppose $\phi\in C_0(G\unit)$ and $f\in\Gamma_c(S,p^*\mcal{A})$. In
order to show $\Phi(\phi)$ is bounded it will suffice 
to show that
$\langle \phi\cdot f, \phi\cdot f \rangle \leq \|\phi\|_\infty^2 \langle f,
f\rangle$.  However, this is equivalent to proving
\[
0 \leq \|\phi\|_\infty^2 \langle f, f \rangle - \langle
\Phi(\phi)f,\Phi(\phi)f\rangle
= \langle \Phi(\|\phi\|_\infty^2 1 - \overline{\phi}\phi)f,f\rangle.
\]
Since general $C^*$-algebraic nonsense assures us that $\|\phi\|_\infty^2 1 -
\overline{\phi}\phi$ is positive in $C_0(S\unit)^1$, it follows
there is some  $\xi \in C_0(S\unit)^1$ such that $\xi^*\xi =
\|\phi\|_\infty^2 1 - \overline{\phi}\phi$.  We now compute
\[
\langle \Phi(\|\phi\|_\infty^2 1 - \overline{\phi}\phi)f,f\rangle  = 
\langle \Phi(\xi^*)\Phi(\xi) f, f\rangle  
= \langle \Phi(\xi)f,\Phi(\xi)f\rangle \geq 0
\]
Hence $\Phi(\phi)$ is bounded and extends to a multiplier on
$A\rtimes S$.  Furthermore, simple calculations show that
$\Phi$ is a nondegenerate
homomorphism from $C_0(S\unit)$ into the center of the multiplier
algebra of $A\rtimes S$.  Thus $A\rtimes S$ is a
$C_0(S\unit)$-algebra.

Let us now address the second part of the proposition.  Fix $u\in
S\unit$ and recall that $A\rtimes S(u) = A\rtimes S/I_u$ where 
\[
I_u = \cspn\{\phi \cdot a : \phi\in C_0(S\unit),a\in A\rtimes S,
\phi(u) = 0\}.
\]
Next, observe that $S$ acts trivially on its unit space so that
$\{u\}$ is a closed $S$-invariant subset in $S\unit$ and $O =
S\unit\setminus \{u\}$ is an open $S$-invariant subset.  It follows from
\cite[Theorem 3.3]{inducpaper} that restriction factors to an
isomorphism from $A\rtimes S/\Ex(O)$ onto $A(u)\rtimes S_u$.  Thus we
will be done if we can show that $I_u = \Ex(O) =
\{f\in\Gamma_c(S,p^*\mcal{A}): \supp f \subset S\setminus S_u\}$.
Given $f\in \Ex(O)$ let $\phi\in C_c(S\unit)$ be zero on $u$ and one
on $p(\supp f)$.  Then $\phi\cdot f = f\in I_u$ and $I_u\subset \Ex(O)$.
Now suppose $f\in I_u$.  Given $\epsilon > 0$ the set $K=\{s : \|f(s)\|
\geq \epsilon\}$ is a compact subset of $\supp f$ and as such we can
find $\phi\in C_c(S\unit)$ such that $\phi$ is one on $p(K)$, zero on
a neighborhood of $u$, and $0\leq \phi\leq 1$.  
It follows quickly that $\phi\cdot f \in
\Ex(O)$ and that $\|\phi\cdot f - f \| < \epsilon$.  Since $\epsilon$
was arbitrary, this is enough to show that $\Ex(O)\subset I_u$.  
\end{proof}

\begin{remark}
One important consequence of Proposition \ref{prop:65} is that the
irreducible representations of $A\rtimes S$ are well behaved.  To
elaborate, \cite[Proposition C.6]{tfb2} states that, as a set,
the spectrum $(A\rtimes S)\sidehat$ can be identified with the
disjoint union $\coprod_{u\in S_u} (A(u)\rtimes S_u)\sidehat$.  In
other words, every irreducible representation of the crossed product
$A\rtimes S$ is lifted from an irreducible covariant
representation of the group crossed product $A(u)\rtimes S_u$ for some
$u\in S\unit$ via restriction on $\Gamma_c(S,p^*\mcal{A})$.  
This fact is at the heart of the analysis in Section \ref{sec:cross-prod-with}.
\end{remark}

We finish this section with a technical lemma.  Recall that given a
$C_0(X)$-algebra $A$ with associated usc-bundle $\mcal{A}$ 
and a locally compact Hausdorff subset $Y\subset
X$ we define $A(Y) := \Gamma_0(Y,\mcal{A})$.

\begin{lemma}
\label{prop:5}
Suppose $(A,S,\alpha)$ is a groupoid dynamical system, $S$ is a group
bundle and $C$ is a closed subset of $S\unit$.  Then $A\rtimes_\alpha
S(C)$ and $A(C)\rtimes_\alpha S|_C$ are isomorphic as
$C_0(C)$-algebras. 
\end{lemma}

\begin{proof}
Since the action of $S$ on its unit space is trivial, both $C$ and $U =
S\unit \setminus C$ are $S$-invariant subsets.  It follows from
\cite[Theorem 3.3]{inducpaper} that restriction factors to an 
isomorphism $\bar{\rho}_1$ of 
$A\rtimes S/\Ex(U)$ onto $A(C)\rtimes S|_C$.  Now let 
\[
I_C = \cspn \{\phi\cdot f : \phi\in C_0(S\unit),
f\in\Gamma_c(S,p^*\mcal{A}), \phi(C)=0 \}.
\]
It follows from some basic $C_0(X)$-algebra theory that the
restriction map $\rho_2:A\rtimes S\rightarrow A\rtimes S(C)$, where we view
both spaces as section algebras of the usc-bundle associated to $A\rtimes
S$, factors to an isomorphism $\bar{\rho}_2:A\rtimes
S/I_C\rightarrow A\rtimes S(C)$.  Similar to the previous proposition, an
approximation argument shows that $I_C=\Ex(U)$, and therefore
we may form the isomorphism $\rho = \bar{\rho}_2 \circ \bar{\rho}_1\inv$ 
of $A(C)\rtimes S|_C$ onto $A\rtimes S(C)$.   The fact that $\rho$ is
$C_0(C)$-linear then follows from a straightforward calculation.  
\end{proof}

\section{Groupoid Crossed Products}
\label{sec:cross-prod-with}

As mentioned in the introduction, we aim to identify the spectrum of
groupoid crossed products via induction and the stabilizer
subgroupoid.  The key to this construction is the following map, which
we will eventually factor to a homeomorphism.  

\begin{prop}
\label{prop:1}
Suppose $(A,G,\alpha)$ is a groupoid dynamical system, that $G$ is 
regular, and that the
isotropy groupoid $S$ has a Haar system.  Then $\Phi:(A\rtimes S)\sidehat \rightarrow (A\rtimes
G)\sidehat$ given by $\Phi(R) = \Ind_S^G R$ is a continuous
surjection. 
\end{prop}

Recall that $A\rtimes S$ is a $C_0(G\unit)$-algebra and that
restriction factors to an isomorphism of $A\rtimes S(u)$ with
$A(u)\rtimes S_u$.  The main difficulty is showing that induction
respects this fibring. 

\begin{lemma}
\label{lem:1b}
Suppose $(A,G,\alpha)$ is a groupoid dynamical system and that the
stabilizer subgroupoid $S$ has a Haar system.  Given $u\in G\unit$ and a
representation $R$ of $A(u)\rtimes S_u$ let $\rho:A\rtimes S
\rightarrow A(u)\rtimes S_u$ be given on $\Gamma_c(S,p^*\mcal{A})$ by
restriction.  Then $\Ind_{S_u}^G R$ is naturally equivalent to
$\Ind_S^G(R\circ \rho)$.  
\end{lemma}

\begin{proof}
The proof of this lemma is relatively straightforward so we shall
limit ourselves to sketching an outline.  Fix $u\in G\unit$ and
suppose $R$ is a representation of $A(u)\rtimes S_u$ on $\mcal{H}$.  Recall from
\cite[Theorem 2.1]{inducpaper} that $\Ind_{S_u}^G R$ acts on the Hilbert tensor
product $\mcal{Z}_{S_u}^G \otimes_{A(u)\rtimes S_u} \mcal{H}$ where
$\mcal{Z}_{S_u}^G$ is a Hilbert $A(u)\rtimes S_u$-module.
Furthermore, recall that $\mcal{Z}_{S_u}^G$ is a completion of
$C_c(G_u,A(u))$.  Similarly $\Ind_S^G(R\circ\rho)$ acts on
$\mcal{Z}_S^G\otimes_{A\rtimes S}\mcal{H}$ where the Hilbert $A\rtimes
S$-module $\mcal{Z}_S^G$ is a completion of
$\Gamma_c(G,s^*\mcal{A})$.  Let
$\pi:\Gamma_c(G,s^*\mcal{A})\rightarrow C_c(G_u,A(u))$ be given by
restriction.  We now define $U:\Gamma_c(G,s^*\mcal{A})\odot
\mcal{H}\rightarrow C_c(G_u,A(u))\odot \mcal{H}$ on elementary tensors
by $U(f\otimes h) = \pi(f)\otimes h$.  It then follows from some
relatively painless calculations that $U$ is isometric and extends to a
unitary map from $\mcal{Z}_S^G\otimes_{A\rtimes S}\mcal{H}$ onto
$\mcal{Z}_{S_u}^G\otimes_{A(u)\rtimes S_u}\mcal{H}$ which intertwines
$\Ind_{S_u}^G R$ and $\Ind_S^G(R\circ\rho)$.  
\end{proof}

\begin{remark}
\label{rem:1}
In light of how natural the unitary intertwining $\Ind_{S_u}^G R$ and
${\Ind_S^G(R\circ \rho)}$ is, we shall often confuse the two.  Furthermore,
since every irreducible representation of $A\rtimes S$ is lifted from a
fibre via restriction, we will feel free to use the notation $\Ind_S^G
R$ even when $R$ is an irreducible representation of $A(u)\rtimes
S_u$ and will interpret $\Ind_S^G R$ as either $\Ind_{S_u}^G R$ or
$\Ind_S^G(R\circ \rho)$ as we see fit.  We trust the reader will
forgive the author for these abuses.  
\end{remark}

The advantage of viewing the induction as occurring on $S$ is that
induction from a fixed algebra is a continuous process. 

\begin{proof}[Proof of Proposition \ref{prop:1}]
As noted above, every irreducible representation of $A\rtimes S$ 
is of the form $R\circ
\rho$ where $R$ is an irreducible representation of $A(u)\rtimes S_u$
for some $u\in G\unit$ and $\rho$ is the canonical extension of the
restriction map.   Since $G$ is regular, we know from
\cite[Proposition 4.13]{inducpaper} that $\Ind_S^G R$ is irreducible.  Thus
$\Phi$ is well defined.  The surjectivity follows immediately
from \cite[Theorem 4.1]{inducpaper}, and the continuity follows from
the fact that Rieffel induction is a continuous process
\cite[Corollary 3.35]{tfb}.
\end{proof}

\subsection{Groupoid Actions}

The goal of this section is to lay groundwork for establishing
the equivalence relation on $(A\rtimes S)\sidehat$ induced by $\Phi$.

\begin{prop}
Suppose $G$ is a locally compact groupoid and that
the isotropy subgroupoid $S$ has a Haar system.  
Then there is a continuous homomorphism $\omega$ from
$G$ to $\R^+$ such that for all $f\in C_c(S)$
\begin{equation}
\label{eq:2b}
\int_S f(s)d\beta^{r(\gamma)}(s) = \omega(\gamma)\int_S f(\gamma s
\gamma\inv)d\beta^{s(\gamma)}(s).
\end{equation}
Furthermore, given $s\in S$ we have $\omega(s) = \Delta^u(s)\inv$ where
$\Delta^u$ is the modular function for the group $S_u$.  
\end{prop}

\begin{proof}
By and large this is proved in the same way as \cite[Lemma
4.1]{ctgIII}.  The only difference is that the stabilizer subgroupoid
$S$ may not be abelian and that, rather than being
$S$-invariant, $\omega(s) = \Delta^u(s)\inv$ for all $s\in S_u$.  This is
shown by the following calculation for $s\in S_{u}$ and $f\in C_c(S)$
\begin{align*}
\omega(s)\inv \int_S f(t)d\beta^u(t) &= \int_S
f(sts\inv)d\beta^{u}(t)  = 
\int_S f(ts\inv)d\beta^u(t) \\
&= \Delta^u(s) \int f(t)d\beta^u(t). 
\end{align*}
Since the remainder of the proof is identical to that of \cite[Lemma
4.1]{ctgIII} we will not reproduce it here.  
\end{proof}

Next we demonstrate the following 
construction which, although we only make use of it
indirectly, is interesting in its own right. 

\begin{prop}
\label{prop:2}
Suppose $(A,G,\alpha)$ is a groupoid dynamical system and that the
isotropy subgroupoid $S$ has a Haar system.  
Then there is an action $\delta$
of $G$ on $A\rtimes_\alpha S$ defined by the collection
$\{\delta_\gamma\}_{\gamma\in G}$ where, for $f\in
  C_c(S_{s(\gamma)},A(s(\gamma)))$, 
\begin{equation}
\label{eq:1}
\delta_\gamma(f)(s) = \omega(\gamma)\inv \alpha_\gamma(f(\gamma\inv s
\gamma)).
\end{equation}
\end{prop}

\begin{proof}
It is easy enough to show that $\delta_\gamma:A(s(\gamma))\rtimes
S_{s(\gamma)}\rightarrow A(r(\gamma))\rtimes S_{r(\gamma)}$ is a well
defined isomorphism and that $\delta$ respects the groupoid operations
on $G$.  The difficult part is proving that the action is continuous.
Suppose $\mcal{E}$ is the usc-bundle associated to the
$C_0(G\unit)$-algebra $A\rtimes S$.  Given $\gamma_n\rightarrow
\gamma_0$ in $G$ and $a_n\rightarrow a$ in $\mcal{E}$ such that
$s(\gamma_n)=p(a_n)=u_n$ for all $n\geq 0$ we must show that
$\delta_{\gamma_n}(a_n)\rightarrow \delta_{\gamma_0}(a_0)$.  Fix
$\epsilon > 0$ and let $v_n = r(\gamma_n)$ for all $n\geq 0$.  First,
choose $b\in A\rtimes S$ such that $b(u_0) = a_0$.  Next, using the
fact that $\Gamma_c(S,p^*\mcal{A})$ is dense in $A\rtimes S$, we can
choose $f\in \Gamma_c(S,p^*\mcal{A})$ such that $\|f(u)-b(u)\| <
\epsilon/2$ for all $u\in G\unit$. Recall that $f(u)$, the image of
$f$ in $A(u)\rtimes S_u$, is exactly the restriction of $f$ to $S_u$.
We now make the following

\begin{claim}
If $f\in\Gamma_c(S,p^*\mcal{A})$ and $\gamma_n\rightarrow \gamma_0$ as
above then $\delta_{\gamma_n}(f(u_n))\rightarrow
\delta_{\gamma_0}(f(u_0))$.
\end{claim}

\begin{proof}[Proof of Claim]
First, suppose $v_n = v_0$ infinitely often.  Then we can pass to a
subsequence, relabel, and assume $v_n = v_0$ for all $n \geq 0$.  Now
suppose we can pass to another subsequence such that for each $n > 0$
there exists $s_n$  with 
\begin{equation}
\label{eq:3}
\|\delta_{\gamma_n}(f(u_n))(s_n) - \delta_{\gamma_0}(f(u_0))(s_n)\|
\geq \epsilon > 0.
\end{equation}
If this is to hold we must either have $\gamma_n\inv s_n\gamma_n \in
\supp f$ infinitely often or $\gamma_0\inv s_n \gamma_0\in \supp f$
infinitely often.  In either case we may pass to a subsequence,
multiply by the appropriate groupoid elements, and find $s_0$ such
that $s_n\rightarrow s_0$.  However, we then have
$f(\gamma_n\inv s_n\gamma_n) \rightarrow f(\gamma_0\inv s_0\gamma_0)$
and $f(\gamma_0\inv s_n \gamma_0) \rightarrow f(\gamma_0\inv s_0
\gamma_0)$.  Since both $\omega$ and $\alpha$ are continuous, it
follows that $\delta_{\gamma_n}(f(u_n))(s_n)$ and
$\delta_{\gamma_0}(f(u_0))(s_n)$ both converge to
$\delta_{\gamma_0}(f(u_0))(s_0)$ and this contradicts \eqref{eq:3}.
It follows quickly that $\delta_{\gamma_n}(f(u_n))\rightarrow
\delta_{\gamma_0}(f(u_0))$ with respect to the inductive limit
topology and thus in $A(v_0)\rtimes S_{v_0}\subset\mcal{E}$.  

Next, suppose that we may remove an initial segment and assume that
$v_n \ne v_0$ for all $n> 0$.  We may also pass to a subsequence,
relabel, and assume that $v_n \ne v_m$ for all $n\ne m$.  Let $K =
\{v_n\}_{n=0}^\infty$.  Then $C = S|_K =
p\inv(K)$ is closed in $S$ and we can define
$\iota$ on $C$ by $\iota(s) = n$ if and only if $p(s) = v_n$.  Some
simple computations then show that the function 
$F(s) = \delta_{\gamma_{\iota(s)}}(f(\iota(s)))(s)$ is continuous and
compactly supported on $C$.   Thus $F\in\Gamma_c(C,p^*\mcal{A})\subset
A(K)\rtimes S|_K$.  It follows from 
Lemma \ref{prop:5} that $A(K)\rtimes S|_K$ is isomorphic
to the restriction $A\rtimes S(K)$.  In particular, we may view $F$ as
a continuous section of $\mcal{E}$ on $K$, where we recall that
$F(v_n)$ denotes the restriction of $F$ to $S_{v_n}$.  Since $F$ is
continuous, we must have $F(v_n)\rightarrow F(v_0)$. However, we
clearly constructed $F$ so that $F(v_n) = \delta_{\gamma_n}(f(u_n))$
for all $n\geq 0$ and this proves our claim. 
\end{proof}

Thus $\delta_{\gamma_n}(f(u_n)) \rightarrow
\delta_{\gamma_0}(f(u_0))$.  Since both $a_n\rightarrow a_0$ and
$b(u_n)\rightarrow a_0$ it follows that $\|a_n-b(u_n)\|\rightarrow 0$
so that eventually 
\[
\|\delta_{\gamma_n}(f(u_n))-\delta_{\gamma_n}(a_n)\| \leq \|f(u_n) -
b(u_n)\| + \|b(u_n)-a_n\| < \epsilon. 
\]
Since $\|\delta_{\gamma_0}(f(u_0))-\delta_{\gamma_0}(a_0)\| =
\|f(u_0)-b(u_0)\| < \epsilon$ by construction, it now follows from
\cite[Proposition C.20]{tfb2} that $\delta_{\gamma_n}(a_n)\rightarrow
\delta_{\gamma_0}(a_0)$ and we are done. 
\end{proof}

The following corollary will eventually
form our foundation for the equivalence classes determined by $\Phi$. 

\begin{corr}
\label{cor:1}
Suppose $(A,G,\alpha)$ is a groupoid dynamical system and that the
stabilizer subgroupoid has a Haar system.  Then the
action $\delta$ induces an action of $G$ on $(A\rtimes S)\sidehat$
given by $\delta\cdot R  = R\circ \delta_\gamma\inv$.  Furthermore, if
$R=\pi\rtimes U$ then $\gamma\cdot R = \rho\rtimes V$ where 
\begin{equation}
\label{eq:4}
\rho(a) = \pi(\alpha_\gamma\inv(a)), \quad\text{and}\quad V_s =
U_{\gamma\inv s \gamma}.
\end{equation}
\end{corr}

\begin{proof}
The fact that the action exists follows immediately from Proposition
\ref{prop:3}.  Calculating that $\rho$ and $V$ are given 
as above is accomplished by composing the canonical injections of
$A(r(\gamma))$ and $S_{r(\gamma)}$ into $M(A(r(\gamma))\rtimes
S_{r(\gamma)})$  with $\gamma\cdot R$.  
\end{proof}

\begin{remark}
We have omitted many of the calculations in these proofs for brevity.
However, enterprising readers wishing to verify the above computations should
make note of the fact that if $\Delta^u$ is the modular function for
$S_u$ then 
\begin{equation}
\label{eq:2}
\Delta^{s(\gamma)}(s) = \Delta^{r(\gamma)}(\gamma s \gamma\inv)
\quad\text{for $\gamma\in G$.}
\end{equation}
\end{remark}

\subsection{Equivalent Representations}

The primary obstacle in working with induced representations is that
they are not very concrete.  The purpose of this section is to
describe a selection of concrete representations which are equivalent
to $\Ind_S^G R$ for a given $R$.  This material is at least inspired
by \cite{ctgIII}, when it doesn't copy it directly.  We begin by
citing the following

\begin{lemma}[{\cite[Lemma 3.2]{ccrgca}}]
\label{lem:32}
Let $G$ be a locally compact Hausdorff groupoid. 
Suppose $u\in G\unit$, that $A$ is a subgroup of $S_u$, and
that $\beta$ is a Haar measure on $A$.  Then the following hold. 
\begin{enumerate}
\item The formula
\[
Q(f)([\gamma]) = \int_A f(\gamma s)d\beta(s)
\]
defines a surjection from $C_c(G)$ onto $C_c(G_u/A)$.  
\item There is a non-negative continuous function $b$ on
  $G_u$ such that for any compact set $K\subset G_u$ the support of
  $b$ and $KA$ have compact intersection and for all $\gamma\in G_u$
\begin{equation}
\label{eq:127}
\int_A b(\gamma s)d\beta(s) = 1.
\end{equation}
\end{enumerate}
\end{lemma}

The function $b$ in Lemma \ref{lem:32} is the normalization of a
function $b'$ which satisfies all of the conditions of (b) except for
\eqref{eq:127}.  This function is guaranteed to exist by \cite[Lemma
1]{dixmiercstar}.  Furthermore, \cite{dixmiercstar} also proves that
$b'$ is positive, continuous, and $b'$ is not zero on any entire
equivalence class.  We now define 
\begin{equation}
\label{eq:3b}
\rho(\gamma) = \int_A b'(\gamma s) \Delta(s)\inv d\beta(s)
\end{equation}
for $\gamma\in G_u$ where $\Delta$ is the modular function for $A$.  
Notice that $\rho(\gamma) > 0$ for all $\gamma$ because the modular function is
strictly greater than zero and $b'$ is positive and not zero on any
entire equivalence class.  An important property of $\rho$ is that 
for $\gamma\in G_u$ and $s \in A$
\begin{align}
\label{eq:4b}
\rho(\gamma s) &= \int_A b'(\gamma st)\Delta(t)\inv d\beta(t) =
\int_A b'(\gamma t)\Delta(s)\Delta(t)\inv d\beta(t) = \Delta(s)\rho(\gamma).
\end{align}
We can now cite the following

\begin{lemma}[{\cite[Lemma 3.3]{ccrgca}}]
\label{lem:4}
There is a Radon measure $\sigma$ on $G_u/A$ such that 
\begin{equation}
\label{eq:129}
\int_G f(\gamma)\rho(\gamma) d\lambda_u(\gamma) = \int_{G_u/A}\int_A f(\gamma
s)d\beta(s)d\sigma([\gamma])
\end{equation}
for all $f\in C_c(G_u)$.  
\end{lemma}

\begin{remark} 
It is not particularly difficult to show that $\sigma$ has full support
on $G_u/A$. 
\end{remark}

Suppose $(A,G,\alpha)$ is a groupoid dynamical system with stabilizer
subgroupoid $S$.  For all $u\in S\unit$ let
$\beta^u$ be a Haar measure on $S_u$.  Using Lemma \ref{lem:4}, for each
$u\in G\unit$ there exists a Radon measure $\sigma^u$ with full support on
$G_u/S_u$ and an associated continuous strictly positive function
$\rho^u$ on $G_u$ such that 
\[
\int_G f(\gamma)\rho^u(\gamma)d\lambda_u(\gamma) = \int_{G_u/S_u}\int_S f(\gamma s)
d\beta^u(s) d\sigma^u([\gamma]).
\]
For the rest of this section whenever we have $(A,G,\alpha)$ and $S$
as above we will let $\sigma = \{\sigma^u\}$ and $\rho = \{\rho^u\}$ be defined in this way.  Next, we construct a
Hilbert space which we will use for one of our equivalent representations. 

\begin{lemma}
\label{lem:2}
Fix $u\in G\unit$ and
suppose $R=\pi\rtimes U$ is a covariant representation of $A(u)\rtimes
S_u$ on a separable Hilbert space $\mcal{H}$.  Let $\mcal{V}_u$ be
the set of Borel functions $\phi:G_u\rightarrow\mcal{H}$ such that
$\phi(\gamma s) = U_s^* \phi(\gamma)$ for all $\gamma\in G_u$ and $s\in
S_u$. Define 
\[
\mcal{L}^2_U(G_u,\mcal{H},\sigma^u) = \left\{ \phi\in \mcal{V}_u :\int_{G_u/S_u}
\|\phi(\gamma)\|^2 d\sigma^u([\gamma]) < \infty\right\}
\]
and let $L^2_U(G_u,\mcal{H},\sigma^u)$ be the quotient of
$\mcal{L}^2_U(G_u,\mcal{H},\sigma^u)$ where we identify functions
which agree almost everywhere.  Then 
$L^2_U(G_u,\mcal{H},\sigma^u)$ is a Hilbert space with the inner
product
\[
(\phi,\psi) := \int_{G_u/S_u}
(\phi(\gamma),\psi(\gamma))d\sigma^u([\gamma]).
\]
\end{lemma}

\begin{proof}
Much of this lemma is straightforward and we will limit ourselves to
proving that $L^2_U(G_u,\mcal{H},\sigma^u)$ is complete.  
Suppose $\phi_n$ is a Cauchy sequence.  We can pass to a
subsequence, relabel and assume that $\|\phi_{n+1} - \phi_n\| < \frac{1}{2^n}$
for all $n$.  We define the
following extended real valued functions on $G_u$ by 
\[
z_n(\gamma) = \sum_{i=1}^n \|\phi_{i+1}(\gamma)-\phi_i(\gamma)\|,\quad\text{and}\quad
z(\gamma) = \sum_{i=1}^\infty \|\phi_{i+1}(\gamma)-\phi_i(\gamma)\|.
\]
Of course, $z_n$ is constant on $S_u$ orbits and factors to a Borel
map on $G_u/S_u$.  Using the triangle inequality in
$L^2(G_u/S_u,\sigma^u)$ we find 
\[
\|z_n\| \leq \sum_{i=1}^n \left(\int_{G_u/S_u}
\|\phi_{i+1}(\gamma)-\phi_i(\gamma)\|^2d\sigma^u([\gamma])\right)^{1/2} =
\sum_{i=1}^n \|\phi_{i+1}-\phi_i\| \leq 1
\]
Since $\|z_n\|^2 = \int_{G_u/S_u} z_n(\gamma)^2d\sigma^u([\gamma])$ it follows from
the Monotone Convergence Theorem that 
$\|z\|^2 = \int_{G_u/S_u} z(\gamma)^2d\sigma^u([\gamma]) \leq 1$.
Hence, there is a $\sigma^u$-null set $N$ such that $[\gamma]\not\in
N$ implies $z(\gamma)< \infty$.  In particular, we can lift $N$ to
$G_u$ and get a $\lambda_u$-null set $NS_u$ such that $\gamma\not\in
NS_u$ implies 
\begin{equation}
\label{eq:31}
\sum_{i=1}^\infty \phi_{i+1}(\gamma)-\phi_i(\gamma)
\end{equation}
is absolutely convergent.  Hence \eqref{eq:31} converges to some
$\phi'(\gamma)\in \mcal{H}$ for all $\gamma\not\in NS_u$.  Furthermore 
\[
\phi'(\gamma) = \lim_{n\rightarrow \infty} \sum_{i=1}^n
\phi_{i+1}(\gamma) - \phi_i(\gamma) = \lim_{n\rightarrow\infty}
\phi_{n+1}(\gamma) - \phi_1(\gamma)
\]
Thus $\phi(\gamma) := \phi'(\gamma)+\phi_1(\gamma)$ satisfies
\begin{equation}
\label{eq:143}
\phi(\gamma) = \lim_{n\rightarrow\infty} \phi_n(\gamma)
\end{equation}
for all $\gamma\not\in NS_u$.  Therefore $\phi_n\rightarrow \phi$ almost
everywhere and $\phi$ is a Borel function.  Now let $\phi$ be zero off $NS_u$.
Then, using \eqref{eq:143} and the fact that $NS_u$ is saturated, we
find that $\phi(\gamma s) = U_s^* \phi(\gamma)$
for all $\gamma\in G_u$ and $s\in S_u$.  Next, given $\epsilon > 0$
there exists $M$ such that $\|\phi_n-\phi_m\| < \epsilon$ for all
$n,m\geq M$.  If $\gamma \not\in NS_u$ then 
$\|\phi(\gamma)-\phi_i(\gamma)\| = \lim_{n\rightarrow
  \infty}\|\phi_n(\gamma)-\phi_i(\gamma)\|$.
Thus, if $k\geq M$, Fatou's Lemma implies that 
$\|\phi-\phi_k\|^2 \leq \liminf_{n\rightarrow\infty}
\|\phi_n-\phi_k\|^2 \leq \epsilon^2.$
Furthermore we have 
\begin{align*}
\|\phi(\gamma)\|^2 &\leq
(\|\phi(\gamma)-\phi_k(\gamma)\|+\|\phi_k(\gamma)\|)^2 \\
&\leq 3\|\phi(\gamma)-\phi_k(\gamma)\|^2 + 3\|\phi_k(\gamma)\|^2
\end{align*}
so that 
\[
\int_{G_u/S_u} \|\phi(\gamma)\|^2 d\sigma^u([\gamma]) \leq 3\|\phi-\phi_k\|^2 +
3\|\phi_k\|^2 < \infty.
\]
Thus $\phi\in\mcal{L}_U^2(G_u,\mcal{H},\sigma^u)$,
$\phi_n\rightarrow \phi$ in $L_U^2(G_u,\mcal{H},\sigma^u)$, and, to
quote the inspiration for this argument \cite[Page 290]{tfb2}, ``this completes the proof of completeness.''
\end{proof}

Using this Hilbert space we have the following 
\begin{prop}
\label{prop:4}
Suppose $(A,G,\alpha)$ is a groupoid dynamical system and that the
stabilizer subgroupoid $S$ has a Haar system.  Fix
$u\in G\unit$ and let $R=\pi\rtimes U$ be a covariant representation of
$A(u)\rtimes S_u$ acting on the separable Hilbert space $\mcal{H}$.
Then $\Ind_{S_u}^G R$ is equivalent to the representation $T^R$ on
$L_U^2(G_u,\mcal{H},\sigma^u)$ defined for $f\in
\Gamma_c(G,r^*\mcal{A})$ and $\phi\in
\mcal{L}^2_U(G_u,\mcal{H},\sigma^u)$ by 
\begin{equation}
\label{eq:7}
T^R(f)\phi(\gamma) = \int_G \pi(\alpha_\gamma\inv(f(\gamma \eta\inv)))
\phi(\eta) \rho^u(\eta)\poshalf \rho^u(\gamma)\neghalf d\lambda_u(\eta).
\end{equation}
\end{prop}

\begin{proof}
First recall that $\Ind_{S_u}^G R$ acts on the Hilbert space
$\mcal{Z}_{S_u}^G \otimes_{A(u)\rtimes S_u}\mcal{H}$ where
$\mcal{Z}_{S_u}^G$ is the completion of the pre-Hilbert $A(u)\rtimes
S_u$-module $C_c(G_u,A(u))$.  Define $V:C_c(G_u,A(u))\odot \mcal{H}
\rightarrow L^2_U(G_u,\mcal{H},\sigma^u)$ on elementary tensors by 
\begin{equation}
\label{eq:5}
V(z\otimes h)(\gamma) = \int_S U_s \pi(z(\gamma s))h \rho^u(\gamma
s)\neghalf \, d\beta^u(s).
\end{equation}
It is not difficult to prove that $V(z\otimes h)$ is an element of
$\mcal{L}_U^2(G_u,\mcal{H},\sigma^u)$.  Furthermore, simple
computations show that $V$ is isometric and extends to an isometry
from $\mcal{Z}_{S_u}^G \otimes_{A(u)\rtimes S_u} \mcal{H}$ into
$L^2_U(G_u,\mcal{H},\sigma^u)$.  In order to show that $V$ is a
unitary it will suffice to show that given
$\phi\in\mcal{L}_U^2(G_u,\mcal{H},\sigma^u)$ such that $(V(z\otimes
h),\phi) = 0$ for all $z\in C_c(G_u,A(u))$ and $h\in\mcal{H}$ then
$\phi$ is zero $\lambda_u$-almost everywhere.  We have
\begin{align}
\label{eq:158}
0 &= (V(z\otimes h),\phi) = \int_{G_u/S_u} (V(z\otimes
h)(\gamma),\phi(\gamma))d\sigma^u([\gamma]) \\ \nonumber
&= \int_{G_u/S_u}\int_S (U_s \pi(z(\gamma s))h,
\phi(\gamma))\rho^u(\gamma s)\neghalf d\beta^u(s)d\sigma^u([\gamma]) \\ \nonumber
&= \int_{G_u/S_u}\int_S (\pi(z(\gamma s))h, \phi(\gamma
s)) \rho^u(\gamma s)\neghalf d\beta^u(s)d\sigma^u([\gamma]) \\ \nonumber
&= \int_G (((\pi\circ z)\otimes h)(\gamma),\phi(\gamma))
\rho^u(\gamma)\poshalf d\lambda_u(\gamma)
\end{align}
where $(\pi\circ z)\otimes h$ denotes the function $\gamma\mapsto
\pi(z(\gamma))h$.  Now suppose $K\subset G_u$ is compact and let
$\phi|_K$ be the function obtained by letting $\phi$ be zero off
$K$.  If  $g\in C_c(G_u)$ is one on $K$ then by Lemma \ref{lem:32}
\[
F([\gamma]) = \int_S g(\gamma s) \rho^u(\gamma s)\inv d\beta^u(s)
\]
defines an element of $C_c(G_u/S_u)$.  We observe that 
\begin{align*}
\int_G \|\phi|_K(\gamma)\|^2 d\lambda_u(\gamma) &\leq
\int_G g(\gamma)\|\phi(\gamma)\|^2 d\lambda_u(\gamma) \\
&= \int_{G_u/H_u} \|\phi(\gamma)\|^2 \int_{S_u} g(\gamma s)
\rho^u(\gamma s)\inv
d\beta^u(s)d\sigma^{u}([\gamma]) \\
&\leq \|\phi\|^2 \|F\|_\infty.
\end{align*}
Thus $\phi|_K\in L^2(G_u,\mcal{H})$. Next, given $z\in C_c(G_u,A(u))$ such that
$\supp z \subset K$ we conclude from \eqref{eq:158} that 
\begin{equation}
\label{eq:8}
0 = \int_G (((\pi\circ z)\otimes h)(\gamma),\phi(\gamma))
\rho^u(\gamma)\poshalf d\lambda_u(\gamma)
= ((\pi\circ z)\otimes h,\phi(\rho^u)\poshalf)_{L^2(K,\mcal{H},\lambda_u)}.
\end{equation}
Because $\rho^u$ is strictly positive, it follows that  
$\phi|_K$ will be zero $\lambda_u$-almost everywhere if we can
show that elements of the form $(\pi\circ z)\otimes h$ span a dense
set in
$L^2(K,\mcal{H},\lambda_u)$.  However, we can restrict ourselves even
further and work with elementary tensors of the form 
\[
f\otimes (\pi(a)h) =((f\otimes a) \circ \pi)\otimes h
\]
 where $f\in C_c(K)$, $a\in A(u)$, and
$h\in\mcal{H}$.  However, using nondegeneracy, it is fairly clear that
these elements span a dense set in
$L^2(K,\mcal{H},\lambda_u)$.  Thus $\phi|_K$ is zero
$\lambda_u$-almost everywhere.  Since $K$ was arbitrary and $G_u$ is
$\sigma$-compact, the result follows.   
Hence $V$ is a unitary and as such we can define the representation
$T^R := V\Ind_{S_u}^G R V^*$.  The fact that $T^R$ is given by
\eqref{eq:7} is the result of a slightly messy computation. 
\end{proof}

Next, because $G_u$ is second countable, we can find a Borel cross
section $c:G_u/S_u \rightarrow G_u$ and this allows us to define a 
Borel map $\delta : G_u\rightarrow S_u$ such that $\gamma =
c([\gamma])\delta(\gamma)$.  We will need these maps in order
to find a representation equivalent to $T^R$ which acts on
$L^2(G_u/S_u,\mcal{H},\sigma^u)$.  

\begin{prop}
\label{prop:6}
Suppose $(A,G,\alpha)$ is a groupoid dynamical system with stabilizer
subgroupoid $S$. Fix $u\in G\unit$, let $R =
\pi\rtimes U$ be a representation of $A(u)\rtimes S_u$ on the
separable Hilbert space $\mcal{H}$, and let $\delta$ be as above.  Then
$T^R$ and $\Ind_{S_u}^G R$ are equivalent to the
representation $N^R$ on $L^2(G_u/S_u,\mcal{H},\sigma^u)$ given by 
\begin{align}
\label{eq:9}
N^R(f)(\phi)([\gamma]) &= \int_G
U_{\delta(\gamma)}\pi(\alpha_\gamma\inv(f(\eta)))U_{\delta(\eta\inv\gamma)}^*
\phi([\eta\inv\gamma]) \ldots \\ \nonumber
&\hspace{.4in}\ldots \rho^u(\eta\inv\gamma)\poshalf \rho^u(\gamma)\neghalf
d\lambda^{r(\gamma)}(\eta)
\end{align}
\end{prop}

\begin{proof}
Define $W:L_U^2(G_u,\mcal{H},\sigma^u)\rightarrow
L^2(G_u/S_u,\mcal{H},\sigma^u)$ by $W(\phi)([\gamma]) =
\phi(c([\gamma]))$ where $c$ is the Borel cross section described
previously.  It follows from a brief computation that $W$ is a
unitary and as such we can use it to define the representation $N^R =
W T^R W^*$.  The fact that $N^R$ is given by \eqref{eq:9} follows from
another computation. 
\end{proof}

\begin{remark}
Before we move forward we need some more measure theoretic trickery.
Observe that because $G_u$ is second countable, the range map
factors to a Borel isomorphism between $G_u/S_u$ and $G\cdot u$.
We use this isomorphism to push the measure $\sigma^u$
forward to a measure $\sigma^u_*$ on $G\cdot u$.  
It is clear that by identifying 
$L^2(G_u/S_u,\mcal{H},\sigma^u)$ and  ${L^2(G\cdot u,\mcal{H},\sigma^u_*)}$
we can view $N^R$ as a representation on the latter
space.  It is easy to see that in this case the action of $N^R$ is given by 
\begin{align*}
N^R(f)(\phi)(\gamma\cdot u) &= \int_G
U_{\delta(\gamma)}\pi(\alpha_\gamma\inv(f(\eta)))U_{\delta(\eta\inv\gamma)}^*
\phi(\eta\inv\gamma\cdot u) \ldots\\
&\hspace{.4in}\ldots
\rho^u(\eta\inv\gamma)\poshalf \rho^u(\gamma)\neghalf
d\lambda^{r(\gamma)}(\eta)
\end{align*}
Since this identification is fairly natural, we won't make much of
a fuss about it.  
\end{remark}

The reason we went through the effort to build $N^R$ is that, as the
next lemma demonstrates, it interfaces nicely with the multiplication
representation of $C^b(G\cdot u)$ on $L^2(G\cdot u,\mcal{H})$.  We
will be able to take advantage of this later on.  

\begin{lemma}
\label{lem:34}
Suppose $(A,G,\alpha)$ is a groupoid dynamical system with stabilizer
subgroupoid $S$.  Let
$u\in G\unit$ and $R=\pi\rtimes U$ be a representation of $A(u)\rtimes
S_u$.  Consider the representation of $C_0(G\unit)$ on $L^2(G\cdot u,
\mcal{H}, \sigma^u_*)$ defined via
\[
N^u(f)\phi(v) = f(v)\phi(v).
\]
Furthermore, given $f\in C_0(G\unit)$ and $g\in \Gamma_c(G,r^*\mcal{A})$
define $f\cdot g(\gamma) := f(r(\gamma))g(\gamma)$.  Then
$N^u(f)N^R(g) = N^R(f\cdot g)$ for all $f\in C_0(G\unit)$ and
$g\in\Gamma_c(G,r^*\mcal{A})$.  
\end{lemma}

\begin{proof}
The representation $N^u$ is nothing more than the restriction map sending
$C_0(G\unit)$ to $C^b(G\cdot u)$ composed with the usual multiplication
representation of $C^b(G\cdot u)$ on $L^2(G\cdot u,\mcal{H})$.  It is
easy to see that if $f$ and $g$ are as above then 
$f\cdot g\in \Gamma_c(G,r^*\mcal{A})$.  The last statement follows
from a computation. 
\end{proof}

We can now prove the following proposition, which tells us that the
equivalence classes on $A\rtimes S$ induced by $\Phi$ are exactly the orbits of
the $G$ action described in Corollary \ref{cor:1}.  

\begin{prop}
\label{prop:104}
Suppose $(A,G,\alpha)$ is a groupoid dynamical system and that the
stabilizer subgroupoid $S$ has a Haar system.  Fix
$u\in G\unit$ and let 
$R$ be an irreducible representation of $A(u)\rtimes S_u$ on a
separable Hilbert space $\mcal{H}$. Then
$\Phi(R)$ is equivalent to $\Phi(\gamma\cdot R)$ for all
$\gamma\in G_u$.   Furthermore, if $G$ is regular and $L$ and $R$
are irreducible representations of $A(u)\rtimes S_u$ and $A(v)\rtimes
S_v$, respectively, 
then $\Phi(L)$ is equivalent to $\Phi(R)$ if and only if there
exists $\gamma\in G_u$ such that $\gamma\cdot L$ is equivalent to $R$.  
\end{prop}

\begin{proof}
Let $R=\pi\rtimes U$ be as above and recall that $\gamma\cdot R =
\rho\rtimes V$ is given by Corollary \ref{cor:1}.  It follows from
Proposition \ref{prop:4} that it suffices to show $T^R$ and
$T^{\gamma\cdot R}$ are equivalent.  Suppose $u=s(\gamma)$,
$v=r(\gamma)$, and define
$W:L^2_U(G_u,\mcal{H},\sigma^u)\rightarrow
L^2_V(G_v,\mcal{H},\sigma^v)$ by 
\[
W(\phi)(\eta) = \omega(\gamma)\poshalf\rho^u(\eta\gamma)\poshalf
\rho^v(\eta)\neghalf f(\eta\gamma)\quad\text{for $\eta\in G_v$.}  
\]
The fact that $W$ is a unitary
which intertwines $T^R$ and $T^{\gamma\cdot R}$ now follows from a
relatively straightforward series of computations. 

\begin{remark}
Those readers wishing to verify these calculations should make note of
the fact that for $\gamma\in G$ as above
\begin{equation}
\label{eq:144}
\int_{G_v/S_v}\phi([\eta\gamma])\omega(\gamma)\rho^u(\eta \gamma)
\rho^v(\eta)\inv d\sigma^v([\eta]) = 
\int_{G_u/S_u}\phi([\eta])d\sigma^u([\eta]).
\end{equation}

\end{remark}

Moving on, suppose $G$ is regular and that we are given $L$ and $R$ as
in the second half of the proposition.  If $\Phi(L)$ is equivalent to
$\Phi(R)$ then it follows from Proposition \ref{prop:6} that $N^R$ is
equivalent to $N^L$.  Let $W$ be the intertwining unitary and let
$N^u$ and $N^v$ be as in Lemma \ref{lem:34}.  We compute 
\begin{align*}
W N^v(f)N^R(g)h &= WN^R(f\cdot g)h = N^L(f\cdot g)Wh \\
&= N^u(f)N^L(g)Wh = N^u(f)WN^R(g)h.
\end{align*}
Since $N^R$ is nondegenerate, this implies that $N^v$ is unitarily
equivalent to $N^u$.  However, if $G\cdot u \cap G\cdot v = \emptyset$
then \cite[Lemma 4.15]{primtrangroup} implies that $N^u$ and $N^v$ can
have no equivalent subrepresentations.  Hence $G\cdot u = G\cdot v$
and there exists $\gamma$ such that $v = \gamma\cdot u$.  Then $R$ and
$\gamma\cdot L$ are both irreducible representations of $A(v)\rtimes
S_v$ and we assumed that $\Phi(R)$ is equivalent to $\Phi(L)$, which is in
turn equivalent to $\Phi(\gamma\cdot L)$ by the above.  It then follows
from \cite[Proposition 4.13]{inducpaper} that $R$ is equivalent
to $\gamma\cdot L$ and we are done. 
\end{proof}

\subsection{Restriction to the Stabilizers}

Now that we know which representations have the same image under
$\Phi$ it is time to show that $\Phi$ is open.  The key construction
is a restriction process from $A\rtimes G$ to $A\rtimes S$.  This is
defined using the following map. 

\begin{prop}
\label{prop:101}
Suppose $(A,G,\alpha)$ is a groupoid dynamical system and the
stabilizer subgroupoid $S$ has a Haar system.  Then
there is a nondegenerate homomorphism $M:A\rtimes S\rightarrow
M(A\rtimes G)$ such that 
\begin{equation}
M(f)g(\gamma) = \int_S
f(s)\alpha_s(g(s\inv\gamma))d\beta^{r(\gamma)}(s)
\end{equation}
for $f\in \Gamma_c(S,p^*\mcal{A})$ and $g\in
\Gamma_c(G,r^*\mcal{A})$.  
\end{prop}

\begin{proof}
Since $M$ is basically defined via convolution it is easy to show that
$M(f)g$ is a continuous compactly supported section.  Some lengthy
computations, which we omit for brevity, show that for 
$f\in \Gamma_c(G,r^*\mcal{A})$ and $g,h\in \Gamma_c(S,p^*\mcal{A})$
\begin{equation}
\label{eq:10}
M(f)(g*h) = M(f)g*h,\quad\text{and}\quad (M(f)g)^* * h = g^* *(M(f^*) h).
\end{equation}
The challenging part is proving the following lemma.  However, since
the proof is long and unenlightening, it has be relegated to the end of
the section.  

\begin{lemma}
\label{lem:3}
The set of functions of the form $M(f)g$ with $f\in \Gamma_c(S,p^*\mcal{A})$ and
$g\in \Gamma_c(G,r^*\mcal{A})$ is dense in $\Gamma_c(G,r^*\mcal{A})$
with respect to the inductive limit topology.  
\end{lemma}

Now, we want to show that $M(f)$ is bounded so that it extends to a
multiplier on $A\rtimes G$. Let $\rho$ be a state on $A\rtimes G$ and define
an inner product on $A\rtimes G$ via 
$(f,g)_\rho = \rho(\langle f, g\rangle)$
where we give $A\rtimes G$ its usual inner-product as an $A\rtimes G$-module.
Let $\mcal{H}_\rho$ be the Hilbert space completion of
$A\rtimes G$ with respect to this pre-inner product.
We would like to apply the Disintegration Theorem \cite[Theorem
7.8]{renaultequiv}
when $\mcal{H}_0$ is the image of $\Gamma_c(G,r^*\mcal{A})$ in
$\mcal{H}_\rho$.  Define $\pi$ on $\mcal{H}_0$ by 
\[
\pi(f)g = M(f)g
\]
for $f\in \Gamma_c(S,p^*\mcal{A})$ and $g\in
\Gamma_c(G,r^*\mcal{A})$. 
It is easy to show that $\pi(f)$ is well defined and that $\pi$ is
a homomorphism from $\Gamma_c(S,p^*\mcal{A})$
to the algebra of linear operators on $\mcal{H}_0$.  
It follows from Lemma \ref{lem:3} that elements of the form $\pi(f)g$ 
are dense in $\mcal{H}_\rho$.
Fix $g,h\in \Gamma_c(G,r^*\mcal{A})$.  We would like to see that
$f\mapsto (\pi(f)g,h)_\rho$ is continuous with respect to the inductive
limit topology.  It suffices to see that the map $f\mapsto M(f)g$ is
continuous with respect to the inductive limit topology and this
is not hard to prove. Finally, the
fact that $(\pi(f)g,h)_\rho = (g,\pi(f^*)h)_\rho$ follows immediately from the
fact that $(M(f)g)^**h = g^**(M(f^*)h)$.  Thus the Disintegration
Theorem implies $\pi$ extends to a representation of
$A\rtimes G$.  In particular, we have 
\[
\rho(\langle M(f)g,M(f)g\rangle) = (\pi(f)g,\pi(f)g)_\rho \leq \|f\|^2
(g,g)_\rho \leq \|f\|^2 \|g\|^2.
\]
By choosing $\rho$ such that $\rho(\langle M(f)g,M(f)g\rangle) =
\|M(f)g\|^2$ we conclude $\|M(f)g\|\leq \|f\|\|g\|$.  Thus $M(f)$
is bounded and it follows from \eqref{eq:10} that $M(f)$
is $A\rtimes G$-linear and  adjointable with adjoint $M(f^*)$.  
Hence $M(f)$ extends to a multiplier on $A\rtimes G$.  What's more, 
$\|M(f)\|\leq \|f\|$ so that $M$ extends to all of 
$A\rtimes S$.  It is then easy to
show that $M$ is a homomorphism on a dense subspace so that it must be
a homomorphism everywhere.  Finally, the fact that $M$ is nondegenerate
follows from Lemma \ref{lem:3}.
\end{proof}

The point is that nondegenerate maps into multiplier algebras yield
continuous restriction processes through the usual general nonsense 
\cite{tfb}, as stated in the following 

\begin{corr}
\label{cor:15}
Suppose $(A,G,\alpha)$ is a groupoid dynamical system and that the
stabilizer subgroupoid $S$ has a Haar system.  
Then there exists a restriction map $\Res_M:\mcal{I}(A\rtimes G)\rightarrow
\mcal{I}(A\rtimes S)$ such that $\Res_M$ is continuous and is
characterized by $\Res_M(\ker R) = \ker \overline{R}\circ M$ for all
representations $R$ of $A\rtimes G$. 
\end{corr}

This next lemma demonstrates the relationship between induction and
this restriction process. 

\begin{lemma}
\label{lem:35}
Suppose $(A,G,\alpha)$ is a groupoid dynamical system and that the
stabilizer subgroupoid $S$ has a Haar system.  Then
given $u\in G\unit$ and an irreducible representation $R$ of 
$A(u)\rtimes S_u$ we have 
\begin{equation}
\label{eq:140}
\Res_M \ker \Ind_{S_u}^G R = \bigcap_{\gamma\in G_u} \ker(\gamma\cdot R).
\end{equation}
\end{lemma}

\begin{proof}
Suppose $R=\pi\rtimes U$ is as above.  Recall from Proposition
\ref{prop:6} that $\Ind_{S_u}^G R$ is equivalent to $N^R$ and let $Q =
\overline{N^R}\circ M$ so that $\Res_M\ker\Ind_{S_u}^G R = \ker
Q$. Now, given $f\in A\rtimes S$ it is
straightforward to show that the collection $\{c([\gamma])\cdot
R(f)\}$ is a Borel field of operators on the trivial bundle
$G_u/S_u\times \mcal{H}$ and that we can form the direct integral
representation $\int_{G_u/S_u}^\oplus c([\gamma])\cdot R\;
d\sigma^u([\gamma])$.  It then follows from a fairly hideous
computation that $Q = \int_{G_u/S_u}^\oplus c([\gamma])\cdot R\;
d\sigma^u([\gamma])$. Hence for $f\in A\rtimes S$ and $\phi\in
\mcal{L}^2(G_u/S_u,\mcal{H},\sigma^u)$ we have
\begin{equation}
\label{eq:133}
Q(f)\phi([\gamma]) = (c([\gamma])\cdot R)(f)\phi([\gamma]).
\end{equation}

Now suppose $f\in A\rtimes S$ and $Q(f) =
0$.  Let $\{g_i\}\in C_c(G_u/S_u)$ be a countable set of functions which
separate points and let $h_j$ be a countable basis for $\mcal{H}$.  
For each $g_i$ and $h_j$ \eqref{eq:133} implies  
\begin{equation}
\label{eq:142}
(c([\gamma])\cdot R)(f)(g_i\otimes h_j)([\gamma]) = 
g_i([\gamma])(c([\gamma])\cdot R)(f)h_j = 0
\end{equation}
for all $[\gamma]\not\in N_{ij}$ where $N_{ij}$ is a $\sigma^u$-null
set.  Let $N = \bigcup_{ij} N_{ij}$ and observe that given
$[\gamma]\not\in N$ \eqref{eq:142} holds for all $i$ and $j$.  In
particular, we can pick $g_i$ so that $g_i([\gamma])\ne 0$ and conclude
that $(c([\gamma])\cdot R)(f) = 0$.  Thus $(c([\gamma])\cdot R)(f) =
0$ for all $[\gamma]\not\in N$.  It then follows from \eqref{eq:4b}
that  $(c([\gamma])\cdot R)(f) = 0$ for $\lambda_u$-almost every 
$\gamma\in G_u$.  

Next, suppose $s\in S_u$.  An elementary computation shows that $R$
and $s\cdot R$ are unitarily equivalent.  In particular
$\gamma\cdot R = c([\gamma])\cdot (\delta(\gamma)\cdot R) \cong
c([\gamma]) \cdot R$ and therefore the previous paragraph implies that
$\gamma\cdot R(f) = 0$ for $\lambda_u$-almost all $\gamma$.  Since $G$
acts continuously on $(A\rtimes S)\sidehat$, the map $\gamma\mapsto
\gamma\cdot R(f)$ is continuous.  Furthermore, $\supp \lambda_u = G_u$ and
$\gamma\cdot R(f) = 0$ for $\lambda_u$-almost every $\gamma\in G_u$ so
that we must have $\gamma\cdot R(f) = 0$ for all $\gamma\in G_u$.  Hence $\ker
Q \subset \bigcap_{\gamma\in G_u} \ker (\gamma\cdot R)$.  The other
inclusion is straightforward.
\end{proof}

We conclude the section with the promised proof of Lemma \ref{lem:3}.

\begin{proof}[Proof of Lemma \ref{lem:3}]
 Fix $\epsilon > 0$ and $g\in
\Gamma_c(G,r^*\mcal{A})$.  Let $K = r(\supp g)$ and choose some fixed
open neighborhood $U$ of $K$ in $S$.  We make the
following claim.  
\begin{claim}
There is a relatively compact open neighborhood $O$ of $K$ in $S$
such that $O\subset U$ and for all $\gamma\in G$ and $s\in O$
\begin{equation}
\label{eq:136}
\|\alpha_s(g(s\inv \gamma)) - g(\gamma)\| < \epsilon/2.
\end{equation}
\end{claim}

\begin{proof}[Proof of Claim]
Suppose not.  Then for every relatively compact neighborhood $W\subset
U$ of $K$ there exists $\gamma_W\in G$ and $s_W \in W$ such that 
\begin{equation}
\label{eq:137}
\|\alpha_{s_W}(g(s_W\inv \gamma_W)) - g(\gamma_W)\| \geq \epsilon/2.
\end{equation}
When we order $W$ by reverse inclusion the sets $\{\gamma_W\}$ and
$\{s_W\}$ form nets in $G$ and $S$ respectively.  In order for
\eqref{eq:137} to hold we must have either $s_W\inv \gamma_W \in\supp g$
or $\gamma_W\in\supp g$ for each $W$.  In either case we have
$r(\gamma_W)\in K$ and, since $W$ is a neighborhood of $K$, 
$\gamma_W \in W\supp g\subset \overline{U}\supp g$.  
Furthermore, $s_W\in W\subset \overline{U}$ for all $W$.  
Since $\overline{U}$ and
$\overline{U}\supp g$ are compact, we can pass to a subnet, twice,
relabel, and find $s\in S$ and $\gamma\in G$ such that 
 $s_W\rightarrow s$ and $\gamma_W\rightarrow
\gamma$.  However, $s_W$ is eventually in every neighborhood of
$K$ so that we must have $s\in K \subset G\unit$.  This implies that $s_W\inv
\gamma_W \rightarrow \gamma_W$.  Using the continuity of the
action, this contradicts \eqref{eq:137}.  
\end{proof}

Let $O$ be the open set from above and choose $f\in C_c(S)^+$ such that
$\supp f \subset O$ and that $\int_S f(s)\beta^u(s) = 1$
for all $u\in K$.  Next, let $\{a_l\}$ be an approximate identity for $A$.
We make the following claim. 
\begin{claim}
There exists $l_0$ such that 
\begin{equation}
\label{eq:138}
\|a_{l_0}(r(\gamma))\alpha_s(g(s\inv \gamma)) - \alpha_s(g(s\inv
\gamma))\| < \epsilon /2 
\end{equation}
for all $s\in \supp f$ and $\gamma \in G$.  
\end{claim}
\begin{proof}[Proof of Claim.]
Suppose not.  Then for each $l$ there exists $\gamma_l\in G$ and
$s_l\in\supp f$ such
that
\begin{equation}
\label{eq:139}
\|a_l(r(\gamma_l))\alpha_{s_l}(g(s_l\inv \gamma_l)) -
\alpha_{s_l}(g(s_l\inv\gamma_l))\| \geq \epsilon /2.
\end{equation}
In order for \eqref{eq:139} to hold we must have
$s_l\inv\gamma_l\in \supp g$ for all $l$. But then $\gamma_l \in
(\supp f)\inv\supp g$.  Since both this set and $\supp f$ are
compact, we can pass through two subnets, relabel, and find $\gamma\in G$
and $s\in S$ such that
$\gamma_l\rightarrow \gamma$ and $s_l\rightarrow s$.  However, we now
have $\alpha_{s_l}(g(s_l\inv\gamma_l))\rightarrow
\alpha_s(g(s\inv\gamma))$.  Choose $b\in A$ such that
$b(r(\gamma)) = \alpha_s(g(s\inv \gamma))$.  Then $a_l
b \rightarrow b$.  
Since $\alpha_{s_l}(g(s_l\inv\gamma_l))\rightarrow b(r(\gamma))$ and 
$b(r(\gamma_l))\rightarrow b(r(\gamma))$, 
we must have $\|\alpha_{s_l}(g(s_l\inv\gamma_l))-b(r(\gamma_l))\| \rightarrow 0$.
Putting everything together, it follows that, eventually,
\begin{align*}
\|a_l(r(\gamma_l))\alpha_{s_l}(g(s_l\inv\gamma_l)) -
\alpha_{s_l}(g(s_l\inv\gamma_l))\| 
&\leq 2\|\alpha_{s_l}(g(s_l\inv \gamma_l)) -
b(r(\gamma_l))\| + \|a_l b - b\| \\
&<\epsilon /2
\end{align*}
and this contradicts \eqref{eq:139}.
\end{proof}

Consider $f\otimes a_{l_0}\in \Gamma_c(S,p^*\mcal{A})$.  First
observe that $\supp f\otimes a_{l_0} \subset U$ and that $U$ was
chosen independently of $\epsilon$.  Next, given $\gamma\in G$ if
$r(\gamma)\not\in K$ then $g(s\gamma) = 0$ for all $s\in
S_{r(\gamma)}$ so that in particular 
\[
M(f\otimes a_{l_0})g(\gamma) - g(\gamma) = \int_S f(s)a_{l_0}(r(\gamma))
\alpha_s(g(s\inv \gamma))d\beta^{r(\gamma)}(s) = 0.
\]
If $r(\gamma)\in K$ then 
\begin{align*}
\|M(f\otimes a_{l_0})&g(\gamma) - g(\gamma)\| \\
=& \left\| \int_S
  f(s)a_{l_0}(r(\gamma))\alpha_s(g(s\inv\gamma))d\beta^{r(\gamma)}(s)
  - \int_S f(s)d\beta^{r(\gamma)}(s) g(\gamma)\right\| \\
\leq& \int_S
  f(s)\|a_{l_0}(r(\gamma))\alpha_s(g(s\inv\gamma))-g(\gamma)\|d\beta^{r(\gamma)}(s)
  \\
\leq& \int_S f(s)
\|a_{l_0}(r(\gamma))\alpha_s(g(s\inv\gamma))-\alpha_s(g(s\inv\gamma))\|
d\beta^{r(\gamma)}(s) \\
&+ \int_S f(s) \|\alpha_s(g(s\inv\gamma)) - g(\gamma)\|
d\beta^{r(\gamma)}(s) \\
&< \epsilon/2 + \epsilon/2 = \epsilon.
\end{align*}
Hence $\|M(f\otimes a_{l_0})g - g\|_\infty < \epsilon$.  This suffices
to show that elements of the form $M(f)g$ are dense in
$\Gamma_c(G,r^*\mcal{A})$ with respect to the inductive limit
topology.  
\end{proof}

\subsection{Identifying the Spectrum}
We have now acquired everything we need to identify the spectrum of
$A\rtimes G$ and prove the main result of the paper. 

\begin{theorem}
\label{thm:crossedstab}
Suppose $(A,G,\alpha)$ is a groupoid dynamical system and that the
isotropy subgroupoid $S$ has a Haar system.  
If $G$ is regular then $\Phi:(A\rtimes
S)\sidehat\rightarrow (A\rtimes G)\sidehat$ defined by $\Phi(R) =
\Ind_S^GR$ is open and factors to a homeomorphism from $(A\rtimes
S)\sidehat/G$ onto $(A\rtimes G)\sidehat$. 
\end{theorem}

\begin{proof}
It follows from Proposition \ref{prop:1} that $\Phi$ is a continuous
surjection and from Proposition \ref{prop:104} that $\Phi$ factors to
a bijection on $(A\rtimes S)\sidehat/G$. All that remains is to show
that $\Phi$ is open.  Suppose
$\Phi(R_i)\rightarrow \Phi(R)$ so that, almost by definition,
$\ker \Phi(R_i)\rightarrow \ker \Phi(R)$.   Using Corollary
\ref{cor:15} we know $\Res_M$ is continuous and therefore 
\[
\Res_M \ker \Phi(R_i) = \Res_M \ker \Ind_S^G R_i \rightarrow
\Res_M \ker \Phi(R) = \Res_M\ker \Ind_S^G R.
\]
Let $u= \sigma(R)$ and $u_i = \sigma(R_i)$ for all $i$ where $\sigma:(A\rtimes
S)\sidehat \rightarrow G\unit$ is the usual map arising from the
$C_0(G\unit)$-action on $A\rtimes S$.  
Using the identifications made in Remark \ref{rem:1}, as well as Lemma
\ref{lem:35}, we have 
\begin{align*}
\Res_M\ker \Ind_S^G R &= \bigcap_{\gamma\in G_u} \ker(\gamma\cdot R),\quad\text{and}\\
\Res_M\ker \Ind_S^G R_i  &= \bigcap_{\gamma\in
  G_{u_i}}\ker(\gamma\cdot R_i)\quad\text{for all $i$.}
\end{align*}
It follows from the definition of the Jacobson topology 
that the closed sets
associated to $\Res_M\ker \Ind_S^G R$  and $\Res_M\ker \Ind_S^G R_i$
are 
\[
F = \overline{\{ \ker \gamma\cdot R : \gamma\in G_u\}},\quad\text{and}\quad
F_i = \overline{\{\ker\gamma\cdot R_i : \gamma \in G_{u_i}\}},
\]
respectively.  Since $\ker R\in F$ it follows from \cite[Lemma 8.38]{tfb2}
that, after passing to a subnet and relabeling, there exists $P_i \in
F_i$ such that $P_i\rightarrow \ker R$.  

Let $\mcal{U}$ be a neighborhood basis of $\ker R$.  
For each $U\in\mcal{U}$ there exists $i_0$ such that $i\geq i_0$ implies
$P_i\in U$.  We let $M:= \{ (U,i) : U\in\mcal{U}, P_i\in U\}$
and direct $M$ by decreasing $U$ and increasing $i$.  Then $M$ is a
subnet of $i$ such that $P_{(U,i)}\in U$ for all $(U,i)\in M$. 
Use this fact to find for each $(U,i)\in M$ some
$\gamma_{(U,i)}\in G_{u_i}$ such that $\ker \gamma_{(U,i)}\cdot R_i \in U$.  
Next, given any $U_0\in \mcal{U}$, choose $i_0$ so
that $P_{i_0}\in U$ and $(U_0,i_0)\in M$.  If $(U,i)\in M$
such that $(U_0,i_0)\leq (U,i)$ then $\ker\gamma_{(U,i)}\cdot R_i \in
U\subset U_0$.  Thus $\ker \gamma_{(U,i)}\cdot R_i\rightarrow \ker R$,
and therefore $\gamma_{(U,i)}\cdot R_i\rightarrow R$.  This suffices to
show that $\Phi$ is open. 
\end{proof}

\begin{remark}
\label{rem:29}
If there is a problem with Theorem \ref{thm:crossedstab} it is that
$(A\rtimes S)\sidehat$ can be just as mysterious as $(A\rtimes
G)\sidehat$.  For instance, if $A$ has Hausdorff spectrum (and is
separable) then each
fibre $A(u)$ can be identified with the compacts. In this case 
$A(u)\rtimes S_u$ is relatively well understood \cite[Section
7.3]{tfb2} and in particular is isomorphic to 
$C^*(S_u,\bar{\omega}_u)$ where $[\omega_u]$ is the Mackey obstruction for
$\alpha|_{S_u}$.  However, even if the stabilizers vary continuously,
the collection $\{\omega_u\}$ may be poorly behaved and identifying
the total space topology of $(A\rtimes S)\sidehat$ may be difficult. 
\end{remark}

The following corollary is immediate and interesting enough to be
worth writing down.  

\begin{corr}
Suppose $(A,G,\alpha)$ is a groupoid dynamical system and that $G$ is a
regular principal groupoid.  Then $(A\rtimes G)\sidehat$ is homeomorphic
to $\widehat{A}/G$. 
\end{corr}

\section{Groupoid Algebras}
\label{sec:groupoid-algebras}

We can use the machinery developed in
Section \ref{sec:cross-prod-with} to prove Theorem
\ref{thm:crossedstab} for certain
non-regular groupoid algebras.  First, we state the following
corollary, which follows immediately from Corollary \ref{cor:1}.  

\begin{corr}
\label{cor:17}
Suppose $G$ is a locally compact Hausdorff groupoid
and that the stabilizer subgroupoid $S$ has a Haar system.  
Then there is a continuous  action of $G$ on
$C^*(S)\sidehat$ given for $\gamma\in G$ and $U\in C^*(S)\sidehat$ by 
\begin{equation}
\label{eq:145}
\gamma\cdot U(s) = U(\gamma\inv s \gamma).
\end{equation}
This action factors to an action of $G$ on $\Prim C^*(S)$.  
\end{corr}

Next, we note that the main result
of \cite{irredreps} states that every representation of $C^*(G)$ induced from
a stability group is irreducible.  Therefore, 
even when $G$ is not regular, we may induce representations from
$C^*(S)\sidehat$ to elements of the spectrum of $C^*(G)$.  What's more, we
obtain the following

\begin{prop}
\label{prop:105}
Let $G$ be a locally compact Hausdorff groupoid and suppose the 
isotropy subgroupoid $S$ has a Haar system.  
Then $\Phi:C^*(S)\sidehat\rightarrow
C^*(G)\sidehat$ defined by $\Phi(U) = \Ind_S^G U$ is
continuous and open.  
\end{prop}

\begin{proof}
It follows from the above discussion that $\Phi$ maps into
$C^*(G)\sidehat$, and
the continuity of $\Phi$ follows from the general theory of Rieffel
induction.  All that is left to do is show $\Phi$ is open.  
Suppose $\Ind U_i \rightarrow \Ind
U$ in $C^*(G)\sidehat$.  Since $\Res_M $ is continuous, it follows
that 
\[
I_i = \Res_M \ker \Ind_S^G U_i \rightarrow I= \Res_M \ker \Ind_S^G U.
\]
Lemma \ref{lem:35} then tells us that 
\[
I = \bigcap_{\gamma\in G_{\hat{p}(U)}} \ker \gamma\cdot U,\quad\text{and}\quad I_i = \bigcap_{\gamma\in
  G_{\hat{p}(U_i)}} \ker \gamma \cdot U_i\quad\text{for all $i$.}
\]
Hence, the closed sets associated to $I$ and $I_i$ are 
\[
F = \overline{\{\ker \gamma\cdot U : \gamma\in G_{\hat{p}(U)}\}},\quad\text{and}\quad 
F_i = \overline{\{\ker \gamma\cdot U_i : \gamma\in G_{\hat{p}(U_i)}\}},
\]
respectively.  
Since $\ker U\in F$ it follows from \cite[Lemma
8.38]{tfb2} that, after passing to a subnet and relabeling, there
exists $P_i \in F_i$ such that
$P_i\rightarrow \ker U$.  It then follows from an argument similar
to that at the end of the proof of Theorem \ref{thm:crossedstab} that
we can pass to a subnet and find $\gamma_i$ such that $\gamma_i\cdot
U_i \rightarrow U$. This suffices to show that $\Phi$ is open. 
\end{proof}

Now, if the stability groups of $G$ are GCR it follows from \cite[Theorem
1.1]{ccrgca} that $C^*(G)$ is Type I or GCR if and only if $G$ is
regular.  Since we are extending Theorem \ref{thm:crossedstab} to
non-regular groupoids, this means potentially working with non-Type I
$C^*$-algebras.  Thus we must use the primitive ideal space
instead of the spectrum.  The
following is an immediate consequence of Proposition \ref{prop:105},
once we extend induction to the primitive ideals in the usual fashion.

\begin{corr}
\label{cor:19}
Let $G$ be a locally compact Hausdorff groupoid and suppose the 
isotropy subgroupoid $S$ has a Haar system.  
Then $\Psi:\Prim C^*(S) \rightarrow
\Prim C^*(G)$ defined by $\Psi(P) = \Ind_S^G P$ is
continuous and open.   
\end{corr}

We would like to factor $\Psi$ to a homeomorphism and to do that 
we will need to get a
handle on the equivalence relation determined by $\Psi$.  

\begin{lemma}
\label{lem:38}
Let $G$ be a locally compact Hausdorff groupoid and suppose the
isotropy subgroupoid $S$ has a Haar system.  
Then $\Psi(P) = \Psi(Q)$ if
and only if $\overline{G\cdot P} = \overline{G\cdot Q}$. 
\end{lemma}

\begin{proof}
Suppose $U,V\in C^*(S)\sidehat$ such that $P = \ker U$ and $Q = \ker
V$.  If $\ker \Ind_S^G V = \ker \Ind_S^G U$ then
$\Res_M \ker \Ind_S^G V = \Res_M \ker \Ind_S^G U$.
However, it now follows from Lemma \ref{lem:35} that 
\[
\bigcap_{\gamma\in G_{\hat{p}(U)}}\gamma\cdot P = 
\bigcap_{\gamma\in G_{\hat{p}(V)}}\gamma\cdot Q
\]
where $\hat{p}$ is the canonical map from $C^*(S)\sidehat$ onto
$S\unit$. 
This implies that the closed sets in $\Prim C^*(S)$ associated
to these ideals must be the same.  Hence $\overline{G\cdot P} =
\overline{G\cdot Q}$.  The reverse direction follows immediately from
the fact that $\Phi$ is continuous and $G$-equivariant.  
\end{proof}

At this point we recall from \cite{geneffhan} that a groupoid is
said to be EH-regular if every
primitive ideal is induced from an isotropy subgroup.  That is, given
$P\in \Prim C^*(G)$ there exists $u\in G\unit$ and $Q\in\Prim
C^*(S_u)$ such that $P = \Ind_{S_u}^G Q$.  Of course, it follows from
\cite[Theorem 4.1]{inducpaper} that regular groupoids are
EH-regular.  In the non-regular case the main result in
\cite[Theorem 2.1]{geneffhan} states that if a groupoid $G$ is amenable in
the sense of Renault \cite{amenable} then $G$ is EH-regular.  
This allows us to give the promised strengthening of Theorem
\ref{thm:crossedstab}.  First, however, recall that
the {\em $T_0$-ization} of a topological space $X$ is the quotient space
$X^{T_0} := X/\sim$ where $x\sim y$ if and only if $\overline{\{x\}} =
\overline{\{y\}}$.  

\begin{theorem}
\label{thm:scalarstab}
Suppose $G$ is a locally compact Hausdorff groupoid
and that the stabilizer subgroupoid $S$ 
has a Haar system.  If $G$ is EH-regular, and in
particular if $G$ is either amenable or regular, then the map
$\Psi:\Prim C^*(S) \rightarrow \Prim C^*(G)$ defined by $\Psi(P) =
\Ind_S^G P$ factors to a homeomorphism
of $\Prim C^*(G)$ with $(\Prim C^*(S)/G)^{T_0}$.  
\end{theorem}

\begin{proof}
It follows from Corollary \ref{cor:19} that $\Psi$ is continuous and
open.  Surjectivity clearly follows from the fact that $G$ is EH-regular.
Finally, it is straightforward to show that $\overline{G\cdot P}
= \overline{G\cdot Q}$ in $\Prim C^*(S)$ if and only if
$\overline{\{G\cdot P\}} = \overline{\{G\cdot Q\}}$ in
$\Prim C^*(S)/G$.  Thus it follows 
from Lemma \ref{lem:38} that the factorization of
$\Psi$ to $(\Prim C^*(S)/G)^{T_0}$ is injective and is therefore a
homeomorphism. 
\end{proof}

\begin{remark}
In the case where $S$ is abelian Theorem \ref{thm:scalarstab} is
particularly concrete because $\Prim C^*(S) = \widehat{S}$ is the dual
bundle \cite{bundleduality} associated to $S$.  
\end{remark}

As in Section \ref{sec:cross-prod-with}, we get the following
corollary, which in this case is a very slight extension of
\cite[Proposition 3.8]{principclark}.

\begin{corr}
If $G$ is an EH-regular,
principal groupoid then $\Prim C^*(G)$ is homeomorphic to $(G\unit/G)^{T_0}$. 
\end{corr}

\bibliographystyle{amsplain}
\bibliography{/home/goehle/references}

\end{document}